 \documentclass[12pt]{amsart}
 \usepackage{tikz}

\usepackage{enumerate}

\usepackage{geometry}                
\usepackage{graphicx}
\usepackage{amssymb}
\usepackage{float}
\usepackage{hyperref} 
\usepackage{bbm}
\usepackage{amsmath,amscd}
\usepackage{algorithm,algorithmic}
\usepackage[english]{babel}
\usepackage[latin1]{inputenc}
\usepackage{times}
\usepackage[T1]{fontenc}
\usepackage{latexsym}
\usepackage{graphics}
\usepackage{color}
\usepackage{cite}
\usepackage{mathrsfs}

\usepackage{amsmath, amssymb, amsthm,latexsym, geometry,graphicx}
\newtheorem{theorem}{Theorem}[section]
\newtheorem{lemma}[theorem]{Lemma}

\newtheorem{corollary}[theorem]{Corollary}

\theoremstyle{definition}

\theoremstyle{remark}

\numberwithin{equation}{section}

\newcommand{\cA}{\mathcal{A}}

\newcommand{\C}{\mathbb{C}}

\newcommand{\N}{\mathbb{N}}

\newcommand{\St}{\mathbb{S}}
\newcommand{\R}{\mathbb{R}}

\newcommand{\sn}{\mathbb{S}^n}
\newcommand{\no}{\noindent}

\newcommand{\la}{\lambda }

\newcommand{\Om}{\Omega }
\newcommand{\ra}{\rightarrow}
\renewcommand{\:}{\colon}

\newcommand{\adjoint}{\operatorname{adj}}
\newcommand{\diver}{\operatorname{div}}

\newcommand{\rn}{\operatorname{\R^n}}
\newcommand{\wR}{\widehat \R}
\newcommand{\sub}{\subseteq}

\newcommand{\CAP}{\operatorname{Cap}}

\newcommand{\defeq}{\mathrel{\mathop:}=}

\begin{document}

\title{The Rickman-Picard Theorem}   
\author{Mario Bonk and Pietro Poggi-Corradini}
\dedicatory{Dedicated to the memory of Juha Heinonen and Seppo Rickman}

\thanks{M.B.\ was partially supported by  NSF grant DMS~1506099. P.P.-C.~was  partially supported by  NSF grant DMS~1515810.}

\date{July 19, 2018}

\address{Department of Mathematics, University of California, Los Angeles, CA 90055, USA.}
\email{mbonk@math.ucla.edu}

\address{Department of Mathematics, Cardwell Hall, Kansas State University,
Manhattan, KS 66506, USA.}
\email{pietro@math.ksu.edu}


\begin{abstract} 
We give a new and conceptually simple  proof of the Rickman-Picard theorem for quasiregular maps based on po\-ten\-tial-theoretic methods. \end{abstract}

\maketitle

\section{Introduction}\label{sec:intro}  

The classical Picard theorem in complex analysis states that a non-constant analytic function defined on the complex plane  $\C$  omits at most one complex value. There are  a dozen or so proofs of this theorem using surprisingly diverse and unexpected approaches. In the 1970s, efforts were made to generalize this theorem to quasiregular maps defined on real Euclidean spaces. Although quasiregular maps (defined below) provide a generalization of analytic maps, many of the proofs that work in the analytic case fail in the higher-dimensional setting. In 1980 Seppo Rickman  \cite{Ri1} was the first to establish  an analog of Picard's theorem for quasiregular maps in higher dimensions. His proof was based  on the concept of   modulus for  path families. Later non-linear potential theory was used to give alternative proofs.  To formulate the Rickman-Picard Theorem, we first review  some basic definitions. 

Let $M$ and $N$ be connected and oriented Riemannian $n$-manifolds, where $n\ge 2$. 
A map $f\: M\ra N$ is called $K$-{\em quasiregular}, where $K\ge 1$, if 
$f$ has distributional derivatives that are $L^n$-integrable (with respect to the Riemannian  measure on $M$) and if the formal differential 
$Df(p)\: T_pM\ra T_{f(p)}N $ satisfies 
$$ \Vert Df(p)\Vert^n \le K \det(Df(p))$$ 
for almost every $p\in M$. Here $\Vert Df(p)\Vert$ denotes operator norm of 
$Df(p)$ with respect  to the norms on the tangent spaces $T_pM$  and $T_{f(p)}N$ 
induced by the Riemannian structures on $M$ and $N$, respectively. 
We call $f\: M\ra N$ {\em quasiregular} if it is $K$-quasiregular for some $K\ge 1$. 

Reshetnyak proved that  a quasiregular map 
 has a continuous representative in its Sobolev class.  In the following,  we will always assume that a 
quasiregular map is continuous, and, in addition,  that it is non-constant.
With these conventions it is known, again by famous results due to Reshetnyak, that a quasiregular map is  open and discrete
(see \cite{Re} for a systematic exposition of Reshetnyak's results on the regularity theory of quasiregular maps).

Let $\rn$ be Euclidean $n$-space, and $\mathbb S^n$ be the unit sphere in $\R^{n+1}$ equipped with the induced Riemannian metric. 
Then  the Rickman-Picard theorem can be formulated as follows
(see \cite[Chapter~4]{Ri}). 

\begin{theorem}\label{rickpic}
 Let $f\:\rn\to\sn$ be a $K$-quasiregular  map. Then $f$ can omit at most $q_0=q_0(n,K)<\infty$ points in $\sn$.
 \end{theorem}
 
Here  the maximal number $q_0(n,K)$ of omitted points
 only depends on $K\ge 1$ and $n\ge 2$ and not on the specific map $f$.   
 In dimension $n=2$ we actually have $q_0(2,K)=2$ independently of $K$. This follows from the classical Picard theorem and the well-known fact that every quasiregular map $f\:\R^2\ra \mathbb{S}^2$ can be written in the form $f=g\circ \varphi$, where 
 $g\: \R^2\ra \mathbb{S}^2$ is analytic and $\varphi\: \R^2\ra \R^2$ is a quasiconformal homeomorphsim. In dimensions $n\ge 3$ we have $q_0(n,K)\to \infty $ as $K\to \infty$, and so Theorem~\ref{rickpic} is qualitatively best possible. This was shown by Rickman for $n=3$ \cite{Ri2} and more recently for arbitrary $n\ge 3$ 
 by Drasin and Pankka \cite{DP}.   
 
In this paper  we give a new and   streamlined proof for the Rickman-Picard
theorem for  quasiregular maps. Our proof  is based on ideas from non-linear potential theory. In contrast to earlier proofs, notably by Eremenko-Lewis \cite{EL} and by Lewis \cite{Le}, we  will not use any results from non-linear potential theory 
established in the literature such as the  Harnack inequality.
We will rely on simple integral inequalities that are fairly easy to establish from first principles.
 
This paper 
is essentially  self-contained except that we take the regularity theory of quasiregular maps and some of their basic topological properties for granted.

  \smallskip \no
{\bf Acknowledgments.} Our work would not have been possible without Seppo Rickman's deep insights into the  geometry 
of quasiregular maps. Many of the basic ideas in this paper 
originated in discussions with our late friend  Juha Heinonen. He would be a co-author if he were still alive. We  dedicate this paper to Juha's  and Seppo's memory.  

\section{Synopsis}
In this introductory section we will give an outline of our proof of the Rickman-Picard Theorem.

The starting point is, as usual, a non-constant $K$-quasiregular map $f\:\rn\to\sn$, $n\ge 2$,  that omits $q$ distinct values $a_1,\dots,a_q$ in $\sn$. The goal is to find a bound on $q$ depending only on the dimension $n$ and the distortion $K$. 

We consider the spherical volume $\sigma_n$ on $\sn\cong\wR^n=\R^n\cup\{\infty\}$ given by the explicit expression
\begin{equation}\label{eq:sign} d\sigma_n(x) = \frac {2^n}{(1+|x|^2)^n}\,   d \lambda_n (x) ,
\end{equation}
where $\la_n$ denotes Lebesgue measure on $\R^n$. The growth behavior of $f$ is controlled
by the  measure $A$ on $\R^n$ obtained by pulling $\sigma_n$ back by $f$; so
\begin{equation}\label{eq:A}
dA(x)= \frac{2^n \det (Df(x))} {(1+|f(x)|^2)^n}\, d\lambda_n(x).  
\end{equation} 

Part of our argument implies that if $q\ge 3$, then $A(\R^n)=\infty$. Actually, one can  show  that if  $f\:\rn\to\sn$ is a quasiregular map with $A(\R^n)<\infty$, then $f$ has an extension to $\infty$ (and can therefore omit at most one value), but   we  will not use this fact. 

The  basic idea now is to show that there are 
constants $C_0>0$ and $C_1=C_1(q,n,K)>0$ such that 
\begin{equation}\label{eq:inflmech}
A(8B)> C_1 A(B), \end{equation}
whenever $B\sub \R^n$ is a ball with $A(B)>C_0$. 
Here one has no good quantitative control for $C_0$, but we will have  $C_1=C_1(q,n,K)>>1$ if $q>>1$.

On the other hand, the following elementary fact is true.

\begin{lemma}[Rickman's  Hunting Lemma] \label{lem:hunt}\hskip -0.2cm\footnote{In  discussions with Juha we jokingly referred to the lemma under this name, because Seppo Rickman  used it to establish  his  Picard Theorem and  in the proof of the lemma a suitable point is ``hunted down". Our proof presented in Section~\ref{sec:prrickpic} is very similar to an argument that can be found in  \cite[p.~85]{Ri}.} Let $\mu$ be a Borel measure on $\R^n$ without atoms 
such that  $\mu(\R^n)=\infty$ and $\mu(B)<\infty$ for each ball $B\sub \R^n$.  Then for some  constant $D=D(n)>1$ the following statement is true:  
 for every  $C>0$ there exists a ball $B=B(a,r)\sub \R^n$ such that
  $$ \mu(B)> C \text{ and } \mu(8B)\le D \mu(B).$$ 
\end{lemma} 
See Section~\ref{sec:prrickpic} for the proof.
A version of this statement  was formulated  in  \cite[Lemma~5.1]{Ri2}. 
If one applies Lemma~\ref{lem:hunt}  to $\mu=A$, then one can derive  a contradiction with
 \eqref{eq:inflmech}, because $C_1(q,n,K)>D(n)$ if $q$ is large enough depending on $n$ and $K$. 

The inequality \eqref{eq:inflmech} should be thought of as an ``inflation mechanism" for the measure $A$
if $q$ is large: if the mass of a ball $B$ exceeds  a certain critical threshold $C_0$, then the mass ``explodes" and increases by a large multiplicative constant $C_1$ if we pass to the eight times  larger ball $8B$. Lemma~\ref{lem:hunt} says that such inflationary behavior is impossible for  measures on $\R^n$ if the measure satisfies some mild conditions and  if $C_1>D(n)$ . 

The main part of the proof is now to set up this inflation mechanism. For this one constructs certain auxiliary functions $v_1, \dots, v_q$ on $\sn$ so that $v_k$ becomes  large only near the omitted value
$a_k$, and pulls these functions back by $f$. 
This idea is standard  and the common  choice is to use functions of the form 
\begin{equation}
v_k(y)=\log^+\frac{\delta}{|y-a_k|}
\end{equation}
with suitable $\delta>0$ (where $a_k\ne \infty$).
These functions  were employed  in the proof of the Rickman-Picard Theorem by Eremenko and Lewis \cite{EL}, for example.  The basic function $\log^+|y|$  is well-known in this context and can be traced back 
to Nevanlinna's theory of value distribution for analytic functions.

In contrast, the elegant Ahlfors-Shimizu variant of value distribution theory (see \cite[Section VI.3]{Ne}) uses the function $\log(1+|y|^2)$
here instead (this  is essentially the K\"ahler potential of the spherical metric on $\widehat\C$). 
Our main new observation is that by using a higher-dimensional analog  of this function,
the potential-theoretic approach to the Rickman-Picard Theorem becomes substantially simpler on a conceptual level.

Namely, for each dimension $n\ge 2$, Lemma~\ref{lem:auxfunc} below guarantees the existence of a radially symmetric  function $v\: \R^n\ra [0,\infty)$ with a logarithmic singularity at $\infty$ such that 
$\Delta_nv=\sigma_n$, where $\Delta_n$ is the $n$-Laplacian and the equation has to be interpreted in the distributional sense. So $v$ is an $n$-subharmonic function with Riesz measure $\sigma_n$. We  then define $v_k=v\circ R_k$, where $R_k$ is a rotation of $\sn$ that moves $a_k$ to  the point $\infty\in \sn\cong \wR^n$. 
Then $v_k$ is still an $n$-subharmonic function with Riesz measure $\sigma_n$, which implies that 
  $u_k=v_k\circ f$ is a non-negative $\mathcal{A}$-subharmonic function with Riesz measure 
  $A$ (see Lemma~\ref{lem:pullharm}). The main point here is that each of these functions has the {\em same} Riesz measure $A$ giving a direct link to the map $f$.
  
Now  it is well-known that the growth of a non-negative subharmonic function is related to the growth of its Riesz measure (see Lemmas~\ref{lem:Riesz<grow} and \ref{lem:grow<Riesz}). Due to the construction of the functions $u_1, \dots, u_q$, the superlevel sets $\{u_k>L_0\}$, $k=1, \dots, q$, are disjoint if $L_0$ is large enough.  On the other hand,  one can show that  if  a ball $B\sub \R^n$ has sufficiently large $A$-mass, then 
$B$ meets all these superlevel sets (Corollary~\ref{cor:meettract}). So these sets  are crowded together near such a ball $B$; if $q$ is large and so there are many such sets, one of them has to be fairly narrow (a precise quantitative version of this crowding phenomenon is given in Lemma~\ref{lem:capest} based on a notion of $n$-capacity). On the other hand, if a non-negative subharmonic function is  supported on a narrow set, then it has to grow fast which in turn drives up its  Riesz measure $A$ (see Lemma~\ref{lem:grow<Riesz}).  In this way, we obtain the desired inflation mechanism 
\eqref{eq:inflmech}, which leads to the proof of Theorem~\ref{rickpic} as we discussed.
For the full details of this argument see Section~\ref{sec:prrickpic}.

Our proof of Theorem~\ref{rickpic} is  fairly  self-contained. In particular, we will not rely on any 
auxiliary results from non-linear potential theory, but will give proofs of all relevant facts. An advantage of our approach is that we do not use any  Harnack-type inequality for 
$\mathcal{A}$-harmonic functions, but we will only use simple integral estimates of 
Caccioppoli-type that are fairly easy to establish. 

  We do take the regularity theory of quasiregular maps for granted though; namely, a quasiregular map $f\: \R^n\ra \R^n$ is continuous, open, and discrete, and its  derivative $Df(x)$ exists and is non-singular for almost every $x\in \R^n$.

\section{$n$-harmonic and $n$-subharmonic functions} \label{sec:nharm}

In this and the following section we will develop the necessary tools from non-linear potential theory. We will prove all the relevant statements,  but we assume that the reader is familiar with 
some basic facts from the theory of Sobolev spaces (see \cite{Zi} for background). 

We use fairly standard notation. We write $X\lesssim Y$ for two quantities 
$X$ and $Y$ if $X\le CY$ for some constant $C\ge 0$ only depending on some ambient parameters specified in the given setting. 

 We denote by $B(a,r)=\{x\in \R^n: |x-a|<r\}$ the open Euclidean ball of radius $r>0$ centered at $a\in \R^n$. If $B=B(a,r)$ is such a ball and $\lambda>0$, then we define 
$\lambda B= B(a,\lambda r)$. If $x=(x_1, \dots, x_n)$ and 
$y=(y_1, \dots , y_n)$ are points in $\R^n$, then 
$$ x\cdot y = x_1 y_1+\dots + x_n y_n$$ 
stands for the standard Euclidean scalar product or dot product. 

If $K\sub \R^n$ is a measurable set and  $1\le p<\infty$, then   $L^p(K)$ denotes  the space of all (equivalence classes of) $L^p$-integrable functions on $K$.
If $\Om\sub \R^n$ is some open set, then $C_c(\Om)$ is the space of all continuous and $C^\infty_c(\Om)$  the space of all $C^\infty$-smooth (real-valued) functions on $\Om$ with compact support. The space $W^{1,n}_{loc}(\Om)$ is the Sobolev space of all  functions $f$ on $\Om$ that are locally $L^n$-integrable with first-order  weak partial derivatives  that are also locally $L^n$-integrable.  We also use this notation for $\R^n$-valued maps $f\: \Om\ra \R^n$ with the understanding that then each component function of $f$ lies in $W^{1,n}_{loc}(\Om)$.

Let $\Om \sub \R^n $ be an open set. 
A function $u\: \Om \ra \R$ is called $n$-{\em harmonic} if its $n$-{\em Laplacian} vanishes, i.e., if 
\begin{equation} \label{eq:nharm} 
\Delta_n u:=\diver(|\nabla u|^{n-2} \nabla u)=0. 
\end{equation}  
Here $\nabla u\: \Om\ra \R^n$ denotes the gradient of $u$ and $\diver V$
the divergence of a vector field $V$. 
This equation arises as the Euler-Lagrange equation for the minimization of the $n$-{\em energy} of $u$ 
given as $$\int_{\Om}|\nabla u|^n. $$ 
Here and in similar integrals below, integration is against  Lebesgue measure 
$\lambda_n$ on $\R^n$ unless otherwise specified. 
 
For \eqref{eq:nharm} to be meaningful, one has to assume that the function $u$ is sufficiently smooth, say $C^2$-smooth.  To allow  more general functions, one can formulate an equivalent definition based on integration against  smooth test functions. Accordingly, we say that a function $u\in W^{1,{n}}_{loc}(\Om)$ is $n$-{\em harmonic}
 if 
\begin{equation} \label{eq:nharm2} 
\int_\Om |\nabla u|^{n-2} \nabla u \cdot \nabla \varphi =0
\end{equation}  
for all $\varphi \in C_c^\infty(\Om)$.  For $C^2$-smooth functions this  definition of an $n$-harmonic
function is equivalent with the one given in \eqref{eq:nharm}. 

We say that $u\in W^{1,{n}}_{loc}(\Omega)$ is $n$-{\em subharmonic} if there exists a positive Borel measure 
$\mu$ on $\Om$ such that 
\begin{equation} \label{eq:subharm} 
-\int_\Om |\nabla u|^{n-2} \nabla u \cdot \nabla \varphi =\int_\Om \varphi\, d\mu 
\end{equation}  for all $\varphi \in C_c^\infty(\Om)$. 
The measure $\mu$ is uniquely determined by $u$ and is called the {\em Riesz measure} of $u$. 
We will write the relation of $u$ and $\mu$ in the  symbolic (or rather distributional) form
$$ \Delta_nu=\mu$$ 
with the understanding that this is interpreted to mean that \eqref{eq:subharm} holds. This property is 
{\em local}, i.e., in order to verify it, it is enough to show that for each $p\in \Om$ the identity \eqref{eq:subharm} is true  for all functions $C_c^\infty(\Om)$ with support in a small neighborhood of $p$. 
An  $n$-subharmonic function is $n$-harmonic precisely if its Riesz measure vanishes identically.

Under conformal maps, $n$-harmonic or $n$-subharmonic functions pull-back to functions of the same type. 
More precisely, let $\Om, \widetilde \Om\sub \R^n$ be open sets, and $f\: \Om \ra \widetilde \Om$ be a (smooth) conformal map. If $\widetilde u\: \widetilde \Om\ra \R$ is $n$-subharmonic, then $u:=\widetilde u\circ f
\: \Om \ra \R$ is also $n$-subharmonic. Moreover, if  $\widetilde \mu$ and $\mu$ are the Riesz measures of 
$\widetilde u$ and $u$, respectively, then $f_*\mu=\widetilde \mu$. 
Here $f_*\mu$ denotes the push-forward of $\mu$ by $f$. This immediately follows from 
\eqref{eq:subharm} and  the transformation formula for integrals. 
 If $\widetilde u$ is $n$-harmonic, then $ \widetilde \mu$ and $\mu$ vanish identically, and so 
$u$ is also $n$-harmonic. This can also be deduced from the conformal invariance of the $n$-energy.

The  locality and  conformal invariance properties of $n$-harmonic and $n$-subhar\-monic functions make it possible to extend these concepts to functions defined on open sets $\Om \sub \wR^n=\R^n\cup \{\infty\}$, possibly containing the point $\infty\in  \wR^n$.  One then verifies 
\eqref{eq:subharm} near $\infty$ after the conformal coordinate change $x\mapsto x/|x|^2$ that sends $\infty$ to $0$. 

We denote by $\sigma_n$ the {\em spherical measure}  on $\wR^n$; so 
$$ d\sigma_n(x) = \frac {2^n}{(1+|x|^2)^n}\,   d \lambda_n (x). $$
We set 
$$\omega_n=\sigma_n(\sn)=\sigma_n(\wR^n)= \frac{2\pi^{(n+1)/2}}{\Gamma((n+1)/2)}. $$ 
The following lemma provides an  auxiliary $n$-subharmonic function on $\R^n$ whose Riesz measure is equal to $\sigma_n$.  

\begin{lemma}\label{lem:auxfunc}
For  fixed $n\ge2$ we define $h\:[0,\infty)\ra 
[0,\infty)$ by setting 
$$ h(r)=\frac{1}{\omega_{n-1}^{1/(n-1)}} \int_0^r \sigma_n(B(0,t))^{1/(n-1)}\,\frac{dt}{t} \quad \text{ for }r\ge 0. $$
Let   $v\: \R^n \ra \R$ be the radially symmetric function given by $v(x):=h(|x|)$ for $x\in \R^n$. Then $v$ and the vector field $V\defeq |\nabla v|^{n-2}\nabla v$  are $C^1$-smooth in $\R^n$, and  we have 
\begin{equation}\label{eq:divV}
 (\diver V)(x) =  \frac {2^n}{(1+|x|^2)^n}, \quad x\in \R^n.
  \end{equation} 
 \end{lemma}
 The last identity   implies that $v$ is $n$-subharmonic and 
satisfies 
 \begin{equation}\label{eq:vdist}  \Delta_n v= \sigma_n 
\end{equation}
in the distributional sense. It follows from the definition of $v$  that $v$ is bounded on compact subset of $\R^n$ and that 
$v(x)\to +\infty$ as $x\to \infty\in \wR^n$.

\begin{proof} For $r>0$ we have  
\begin{equation}\label{eq:sigma} 
\sigma_n(B(0,r))=2^n \omega_{n-1} \int_0^r\frac{t^{n-1}}{(1+t^2)^n}\,dt.
\end{equation}
From this it easily follows that the function $h$ in the statement is $C^\infty$-smooth on $(0,\infty)$. This in turn implies that   $v$ is  $C^\infty$-smooth
on $\R^{n}\backslash \{0\}$.   In order to investigate the behavior of $v$ near $0$, we expand the integral in \eqref{eq:sigma} 
in a power series near $0$. Then 
for $r>0$ sufficiently close to $0$ we have 
\[\label{eq:sigexp} 
\sigma_n(B(0,r))=\frac{2^n}{n} \omega_{n-1} r^n (1+O(r^2)),
\]
where here and below $1+O(r^2)$ indicates a power series in $r^2$ that converges near  $0$. 

This and the definition of $h$ in turn imply that for $r\ge 0$ close to $0$ we have 
$$h(r)=a_n r^{n/(n-1)} (1+O(r^2)), \text{ and } h'(r)= b_n r ^{1/(n-1)} (1+O(r^2))$$
with  some constants  $a_n,b_n>0$. 
 Hence, $h^\prime$ exists and is continuous on $[0,\infty)$ with $h^\prime(0)=0$.
 Since  $v(x)=h(r)$ with $r=|x|$  for $x\in \R^n$, it easily follows that $v$ is differentiable at $x=0$ with $\nabla v(0)=0$ and that 
$$ \nabla v(x)=  \frac{h'(r)}{r}x $$ 
is a continuous function of  $x\in \R^n$. Hence $v$ is indeed $C^1$-smooth on $\R^n$. 

Note that 
\begin{equation}\label{eq:V}
V(x) = |\nabla v(x)|^{n-2}\nabla v(x) = \frac{h'(r)^{n-1}}{r}x
=\frac{\sigma(B(0,r))}{\omega_{n-1} r^{n}} x 
\end{equation} 
for $x\in  \R^n\backslash \{0\}$. 
This implies that $V$ is  $C^\infty$-smooth on 
$\R^n\backslash\{0\}$. To investigate the behavior of $V$ near $0$, we define 
$$\phi(r)\defeq \frac{\sigma(B(0,r))}{\omega_{n-1} r^n}.$$
Then by \eqref{eq:sigma}, for $r\ge 0$ near $0$ we have
$$ \phi(r)= c_n  (1+O(r^2))$$
with  some constant $c_n>0$. In particular, $x\mapsto \phi(|x|)$ is $C^1$-smooth (actually $C^\infty$-smooth) near $0$. 
Since $V(x)=\phi(r) x$, this implies that $V$ is  $C^1$-smooth near $0$, and so $C^1$-smooth on $\R^n$.

%


In order to show \eqref{eq:divV}, 
it is enough to verify the distributional version of this identity, namely
that 
$$ -\int  V \cdot \nabla \varphi =\int \varphi \, d\sigma_n $$  
for each $\varphi\in C^\infty_c(\R^n)$. 
Indeed,  if $\varphi \in C^\infty_c(\R^n)$ 
is arbitrary, then a computation in polar coordinates based on \eqref{eq:V} and \eqref{eq:sigma} shows that 
\begin{align*}
-\int V \cdot \nabla\varphi & = - \int_{\St^{n-1}}\biggl( \int_0^\infty
\frac{\sigma_n(B(0,r))}{\omega_{n-1}} \frac{\partial \varphi(r\xi)}{\partial r} \, dr \biggr) d \sigma_{n-1}(\xi)\\
&= \frac{1}{\omega_{n-1}}\int_{\St^{n-1}}\biggl( \int_0^\infty
\frac{d \sigma_n(B(0,r))}{dr}   \varphi(r\xi) \, dr \biggr) d \sigma_{n-1}(\xi)\\
&= \int_{\St^{n-1}}\biggl( \int_0^\infty
  \varphi(r\xi)\frac{2^n r^{n-1}}{(1+r^2)^n} \, dr \biggr) d \sigma_{n-1}(\xi) \\
 &=\int \varphi \, d \sigma_n.  
\end{align*} 
The identity \eqref{eq:divV} follows.   \end{proof}

If $A$ is a (real) $n\times n$-matrix, we denote the {\em adjunct matrix} of $A$ by 
$\adjoint A$. This is an $n\times n$-matrix, whose entries are given by 
$$ (\adjoint A)_{ij}= (-1)^{i+j}\det A_{ji}$$ 
for $i,j=1, \dots, n$, where $A_{ji}$ is the minor of $A$ obtained by deleting the $i$th column and the $j$th row  of $A$. Note that 
$$ A (\adjoint A)= (\det A) I_n, $$ 
where $I_n$ is the $n\times n$-unit matrix.

\begin{lemma}\label{lem:divergence} Let $\Om,\Om'\sub \R^n$ be
  open sets in $\R^n$, $f\: \Om \ra \Om'$ be a continuous map in $W^{1,n}_{loc}(\Om)$, and $V\: \Om' \ra \R^n$ be a $C^1$-smooth vector field on $\Om'$. Then 
\begin{equation}\label{eq:adjdiv}
\int_\Om [(\adjoint Df)(V\circ f)]\cdot \nabla \varphi =-\int_\Om \varphi [(\diver V)\circ f] J_f
\end{equation} for all $\varphi\in C_c^\infty(\Om)$. 
\end{lemma}

Here $J_f =\det(Df)$ denotes the Jacobian of $f$ and  $\adjoint Df$  the matrix-valued function $x \mapsto  \adjoint Df(x)$. Both are
defined   almost everywhere on $\Om$.  Moreover,   
$(\adjoint Df)(V\circ f)$ is the $\R^n$-valued function obtained from  
pointwise multiplication of  the matrix $(\adjoint Df)(x)$ with the column vector $(V\circ f)(x)$. 
\begin{proof} Suppose first that  $f$ is $C^\infty$-smooth. If  we define 
$$W:= \varphi [(\adjoint Df)(V\circ f)] , $$ then $W$  is a $C^1$-smooth vector field on $\Om$ with compact support in $\Omega$, and so 
$$\int_\Om \diver W=0. $$ 
Now 
 \begin{equation}\label{eq:Wdiv}
 \diver  [(\adjoint Df)(V\circ f)]= [(\diver V)\circ f] J_f. 
 \end{equation} 
This is straightforward to establish  by direct computation if one uses the easily verified (and well-known) fact that  each column of the matrix function $\adjoint Df$, considered as a vector field, is  divergence free. 
We obtain $$\diver W= [(\adjoint Df)(V\circ f)]\cdot \nabla \varphi + \varphi [(\diver V)\circ f] J_f,$$
 and  \eqref{eq:adjdiv} follows if $f$ is $C^\infty$-smooth.

The general case can be derived from the smooth case by   an approximation argument. Namely, if  $f\in W^{1,n}_{loc}(\Om)$ we can find $C^\infty$-smooth maps  $f_k\: \Om\ra  \R^n$ such that on the support $K$ of $\varphi$
we have uniform convergence $f_k\to f$, and convergence $Df_k\to Df$ in $L^n(K)$. Then $J_{f_k}\to J_f$ and $\adjoint Df_k \to \adjoint Df$  in $L^1(K)$. By the first part of the proof,   \eqref{eq:adjdiv}
is true for $f_k$ with $k$ large; by taking the limit $k\to \infty$ on both sides, we obtain   \eqref{eq:adjdiv} for the map $f$.  \end{proof}

Using the language of differential forms, one can outline a more conceptual  way of verifying the crucial identity \eqref{eq:Wdiv}  as follows. We can identify the vector field $V$ with a $C^1$-smooth $(n-1)$-form 
$\alpha$ on $\Om'$ with the same $n$ components up to sign in standard coordinates on $\R^n$.  If $d$ denotes exterior differentiation of forms,  $\text{Vol}$ the standard volume form on $\R^n$, then we can choose these signs so that  
$$ d\alpha= (\diver V) \text{Vol}. $$ Moreover, if 
$f^*$ denotes the pull-back operation on forms by $f$, then 
$ f^*\alpha$ is an $(n-1)$-form whose coefficients in standard coordinates correspond to the components of the vector field $(\adjoint Df)(V\circ f)$. Since $f^*\circ d=d \circ f^*$ for smooth $f$, we obtain  
\begin{align*} \label{eq:forms}
\diver [(\adjoint Df)(V\circ f)]\text{Vol} &= d( f^*\alpha)=
   f^*(d\alpha)\\ &=   f^*((\diver V) \text{Vol} )\notag 
= [(\diver V)\circ f]  J_f \text{Vol}. \end{align*} 
 Equation \eqref{eq:Wdiv} follows.

\section{$\mathcal{A}$-harmonic and $\mathcal{A}$-subharmonic functions}

Pull-backs of $n$-harmonic or $n$-subharmonic functions by quasiregular maps are  not  of the same type in general, but one obtains functions that still satisfy a non-linear degenerate  elliptic equation of divergence type. This is well-known and the basis of the  potential-theoretic method  to  investigate  quasiregular maps.
Here we only discuss some basic facts relevant for  our approach. 

Let $\mathcal A\: \R^n \times \R^n\ra \R^n$ be a measurable map such that for almost every $p\in \R^n$ the map
$\xi \mapsto  \mathcal A(p, \xi)$ is defined on $\R^n$ and satisfies 
\begin{itemize}
\smallskip 
\item[\textnormal{(i)}] $\mathcal {A}(p, \lambda \xi)= |\lambda|^{n-2} \lambda \mathcal A(p,  \xi)$, 

\smallskip 
\item[\textnormal{(ii)}] $|\mathcal A(p,\xi)| \le  c_1 |\xi|^{n-1}$,  

\smallskip 
\item[\textnormal{(iii)}]   $\mathcal A(p,\xi) \cdot \xi  \ge  c_2 |\xi|^{n}$
\end{itemize}
for all $\xi\in \R^n$ and $\lambda\in \R$. 
Here $c_1$ and $c_2$ are positive constants independent of $p$ and $\xi$. 
These requirements are modeled on properties of the basic example  $\mathcal A(p, \xi)=|\xi|^{n-2} \xi$.

Let  $f\: \R^n \ra \R^n$ be  a $K$-quasiregular map. Then $Df(p)$ exists for almost $p\in \R^n$ and is an invertible linear map on $\R^n$. We identify $Df(p)$ with the Jacobi matrix (the matrix representation of $Df(p)$ with respect to the standard basis on $\R^n$). Then 
 \begin{equation}\label{def:G}
 G(p) := \det(Df(p))^{2/n} Df(p)^{-1} (Df(p)^{-1})^t 
  \end{equation}  
 is defined for almost every $p\in \R^n$. Here $A^{-1}$ and $A^t$ indicate the inverse and the transpose of a  matrix $A$, respectively.   Since $f$ is  a $K$-quasiregular map, we have 
 $$ \frac 1C |\xi|^2 \le G(p)\xi \cdot \xi \le C|\xi|^2 $$
 for almost every $p\in \R^n$ and all $\xi \in \R^n$,  where $C=C(n,K)>0$  only depends on $n$ and $K$. 
 
 If we define  
 \begin{equation}\label{def:A}  \mathcal {A}(p,  \xi)= (G(p)\xi \cdot \xi)^{(n-2)/2} G(p) \xi, 
 \end{equation} 
then $\mathcal{A}$ has the above properties (i)--(iii) with constants $c_1=c_1(n,K)$ and $c_2=c_2(n,K)$.

Suppose  $\mathcal{A}$ satisfying (i)--(iii) is given, and $u \in W^{1,n}_{loc}(\R^n)$. For ease of notation we write 
$\mathcal{A}_u$ for the almost everywhere defined  measurable function 
$x\in \R^n \mapsto  \mathcal{A}(x, \nabla u(x))$ with values in $\R^n$.
 We say that  $u\:\R^n\ra \R$  is 
$\mathcal{A}$-{\em subharmonic} if $u\in W^{1,n}_{loc}(\R^n) $ and if there exists a positive measure $\mu$ on $\R^n$ (the {\em Riesz measure} of $u$) 
such that 
\begin{equation}\label{eq:subweak}
 -\int \mathcal{A}_u \cdot \nabla\varphi = \int \varphi \, d \mu
 \end{equation}
for all $\varphi\in C_c^\infty(\R^n)$. In other words, $u$ satisfies the 
equation 
$$ \diver \mathcal{A}_u=\mu$$
in the distributional sense.  For our purposes it is actually enough   to  only consider continuous  $\mathcal{A}$-subharmonic functions.  A standard approximation argument shows that \eqref{eq:subweak} remains valid for all functions $\varphi \in C_c(\R^n)\cap W_{loc}^{1,n}(\R^n)$.

\begin{lemma}[Caccioppoli inequality] \label{lem:cacc}
Let $u\:\R^n \ra [0,\infty)$ be a non-negative continuous $\mathcal{A}$-subharmonic function.
Then
$$ \int  \varphi^n  |\nabla u|^n \le C \int u^n |\nabla \varphi|^n $$ 
for all non-negative $\varphi\in C_c^\infty(\R^n)$, where $C=C(n,c_1, c_2)>0$. \end{lemma}

\begin{proof} In the following proof all implicit multiplicative constants are positive and only depend 
on  $n$, $c_1$, and $c_2$. 
If $\mu$ the the Riesz measure of $u$, then for each non-negative function $\psi\in 
  C_c(\R^n)\cap W_{loc}^{1,n}(\R^n)$ we have 
$$ \int \mathcal{A}_u\cdot \nabla\psi= -\int \psi \, d\mu  \le 0.$$

Now let $\varphi\in C_c^\infty(\R^n)$ with $\varphi\ge 0$ be arbitrary, and choose $\psi= \varphi^n u$ in the previous inequality. This is possible, because $u\ge 0$ and $u\in C_c(\R^n)\cap W_{loc}^{1,n}(\R^n)$. Then 
$$\nabla \psi = n u \varphi^{n-1} \nabla \varphi+  \varphi^n \nabla u, $$ 
and so 
\begin{align*} 
\int \varphi^n  \mathcal {A}_u\cdot \nabla u  &=
\int \mathcal {A}_u \cdot \nabla \psi - n
\int u\varphi^{n-1} \mathcal {A}_u\cdot \nabla  \varphi \\
&\le - n
\int u\varphi^{n-1} \mathcal {A}_u\cdot \nabla  \varphi \le 
n \int u\varphi^{n-1} |\mathcal {A}_u| \cdot |\nabla \varphi |. 
\end{align*}

Using this and the properties (ii) and (iii) of $\mathcal{A}$, we obtain
\begin{align*} 
0\le \int   \varphi^n |\nabla u|^n &\lesssim \int   
\varphi^n\mathcal {A}_u \cdot \nabla u \\
&\lesssim \int   u  \varphi^{n-1} |\mathcal {A}_u| \cdot  |\nabla \varphi |  \\
&\lesssim \int  u \varphi^{n-1}  |\nabla u|^{n-1}  |\nabla \varphi|  \\
&\le  \biggl(\int  \varphi^{n}  |\nabla u|^{n}\biggr)^{(n-1)/n} \biggl(\int u^n  |\nabla \varphi|^{n} \biggr)^{1/n}. \end{align*}
The desired inequality follows. 
\end{proof} 

\begin{lemma}[Growth controls  Riesz measure] \label{lem:Riesz<grow}Let $u\:\R^n \ra [0,\infty)$ be a non-negative continuous $\mathcal{A}$-subharmonic function
with Riesz measure $\mu$, and $B=B(a,r)$, where $a\in \R^n$ and $r>0$. Then 
\begin{equation}\label{eq:Riesz<grow}
 \mu(B) \le C\sup_{x\in 2B} u(x)^{n-1} ,
 \end{equation}  where $C=C(n,c_1, c_2)>0$. \end{lemma}

 \begin{proof}  In the following,  all implicit multiplicative constants are again  positive and only depend 
on  $n$, $c_1$, and $c_2$. 

We can pick a function $\varphi \in C^\infty_c(\R^n)$ such that 
$$0\le \varphi  \le 1, \ \varphi |B= 1,\  \varphi  |\R^n\backslash  2B= 0, \ |\nabla \varphi |\lesssim 1/r. $$ 
Then $\displaystyle \int   |\nabla \varphi|^{n}\lesssim 1. $ 
Applying \eqref{eq:subweak} and the Caccioppoli inequality, we conclude that 
\begin{align*} 
\mu(B)&\le \int \varphi^n \, d\mu =-n\int \varphi^{n-1} \mathcal {A}_u\cdot   \nabla \varphi 
 \\
&\lesssim \int  \varphi^{n-1}   |\nabla u|^{n-1}  |\nabla \varphi|  \\
&\le  \biggl(\int \varphi^{n}  |\nabla u|^{n} \biggr)^{(n-1)/n} \biggl(\int     |\nabla \varphi|^{n} \biggr)^{1/n}\\
&
\lesssim  \biggl(\int \varphi^{n}  |\nabla u|^{n} \biggr)^{(n-1)/n} \\
&\lesssim  \biggl(\int  u^n |\nabla \varphi|^{n} \biggr)^{(n-1)/n} \\
&\le \biggl( \sup_{x\in 2B} u(x)^{n-1} \biggr)  \biggl(\int   |\nabla \varphi|^{n} \biggr)^{(n-1)/n} 
\lesssim  \sup_{x\in 2B} u(x)^{n-1}. 
\end{align*} \end{proof}

\begin{lemma} \label{lem:cacccon1}
Let $u\:\R^n \ra [0,\infty)$ be a non-negative   continuous  $\mathcal{A}$-subharmonic 
function. If $u$ is bounded, then it is a constant function. 
\end{lemma} 

\begin{proof} If $u$ is bounded, then  
 by Lemma~\ref{lem:cacc} we can find  $C\ge 0$ such that 
$$ \int |\nabla u|^n \varphi^n \le C \int  |\nabla \varphi|^n $$ 
for all $\varphi\in C_c^\infty(\R^n)$. 
Now it is a well-known fact that for each $r\ge 0$ there exist functions  $\varphi\in C_c^\infty(\R^n)$ with $\varphi\ge 0$, $\varphi|B(0,r)=1$, and $\int  |\nabla \varphi|^n$ arbitrarily small
(because  the $n$-capacity of $\infty$ vanishes  in $\R^n$). This implies that 
$$\int_{B(0,r)}  |\nabla u|^n=0$$ for each $r\ge 0$, and so $\nabla u=0$ almost everywhere on $\R^n$. It follows that $u$ is equal to a  constant function (see \cite[Corollary 2.1.9]{Zi}). 
\end{proof} 

If $g\: \R^n\ra \R$ is any function, we use some obvious  notation for sets related to sub- and superlevels  of $g$. For example, if $a,b\in \R$ and $a\le b$, then 
$\{a\le g\le b\} \defeq \{x \in \R^n: a\le g(x)\le b\}$, $\{ g>b\} \defeq\{x\in \R^n: u(x)>b\}$, etc. 

We require the following fact for the superlevel sets of  an 
$\mathcal{A}$-sub\-har\-monic 
function. 

  \begin{lemma} \label{lem:cacccon2}
Let $u\:\R^n \ra [0,\infty)$ be a non-constant,  non-negative, and continuous
  $\mathcal{A}$-subharmonic 
function. Then for each $L\ge 0$ the open set $\{u>L\}$ has no bounded 
 components. 
\end{lemma} 

\begin{proof} Let $L\ge 0$ be arbitrary. Then  the set $\{u>L\}$ is open, and it is non-empty by 
Lemma~\ref{lem:cacccon1}.
 We argue by contradiction and assume that 
$\{u>L\}$ has a (non-empty) bounded component $\Om$. Then $u|\partial \Om= L$.
Define $\varphi=\max\{(u-L), 0\}$ on $\Om$ and $\varphi=0$ on $\R^n\backslash  \Om$.  Then 
$\varphi$ is a non-negative  function in $C_c(\R^n)\cap W^{1,n}_{loc}(\R^n)$ with 
$\nabla \varphi=\nabla u$ on $\Om$ and $\nabla \varphi =0$  on $\R^n\backslash  \Om$
(since  $u$ is absolutely on almost every line, one easily checks that $\nabla \varphi$ defined in this way is indeed a distributional  gradient of $\varphi$; see \cite[Corollary 2.1.8]{Zi}).

 If we denote by $\mu$  the Riesz measure of $u$, then we can apply \eqref{eq:subweak}  for this function $\varphi$.
Hence 
$$
\int_\Om \mathcal{A}_u\cdot \nabla u =
\int \mathcal{A}_u\cdot \nabla\varphi =
 -\int \varphi \, d \mu\le 0. 
$$
On the other hand,
$$ \mathcal{A}_u\cdot \nabla u \ge c_2 |\nabla u|^n\ge 0$$ 
almost everywhere on $\R^n$.  This is only possible if $\nabla u(x)=0$ for almost every $x\in \Om$. Therefore,  $u$ is locally constant on $\Om$, and hence constant on $\Om$, because $\Om$ is open and connected. 
Since $u=L$ on $\partial \Om$, this implies that $u=L$ on $\Om$; but we know that $u>L$ on $\Om$. This is a contradiction.
\end{proof}

\section{Proof of the Rickman-Picard theorem} 
\label{sec:prrickpic}

We first provide a proof of Rickman's Hunting Lemma.  
\begin{proof}[Proof of Lemma~\ref{lem:hunt}] We can find a number $D=D(n)\in \N$ so that every ball $B\sub \R^n$ of radius $\rho>0$ can be covered by $D$ balls of radius  $ \rho/16$ centered in $B$
(this is true, because $\R^n$ is a ``doubling" metric space). 

  We argue by contradiction and assume that there is a constant $C>0$ such that for every ball $B\sub \R^n$ with $\mu(B)> C$ we have $\mu (8B)>D \mu(B)$.

Pick $x_0=0$. Since $\mu(\R^n)=\infty$ there exists 
$r>0$ such that for $B_0:=B(x_0, r)$ we have $\mu(B_0)>C$. 
Then by our hypotheses  $\mu(8B_0)> D \mu(B_0)$. By choice of $D$, the ball $8B_0$, which has radius $8r$, can be covered by $D$ balls of radius $r/2$ centered at points in  $8B_0$. Hence there exists $x_1\in \R^n$ with $|x_1-x_0|<8r$ 
such that 
$$ \mu (B(x_1, r/2))\ge \frac 1D \mu  (B(x_0, 8r))\ge \mu(B(x_0, r))> C. $$
Now we repeat the argument for $B_1:= B(x_1, r/2)$, and so on,  decreasing the radii of the balls by the factor $2$ in each step. In this way, 
 we obtain a sequence of points $x_k$, $k\in \N_0$, such that 
for all $k\in \N_0$ we have 
$$ |x_{k+1}-x_k|<  2^{3-k} r $$  and $\mu(B_k)>C$, where 
$B_k:=B(x_k, 2^{-k}r)$. 

The points $x_k$ form a Cauchy sequence and so there exists $x_\infty\in \R^n$ 
such that $x_k \ra x_\infty$ as $k\to \infty$. 
If $\delta>0$ is  arbitrary, then 
$B_k \sub B(x_\infty, \delta)$ for some  $k\in \N$, and so  
 $\mu(B(x_\infty, \delta))\ge \mu(B_k)>C$.
Hence 
$$ \mu(\{x_0\})=\lim_{\delta\to 0} \mu(B(x_\infty, \delta))\ge C>0. $$
Here we used that  $\mu(B(x_\infty, \delta))<\infty$ as follows from our hypotheses. 
We obtain a contradiction, because $\mu(\{x_0\})>0$, but $\mu$ has no atoms. 
\end{proof}

To set  up the proof of the Rickman-Picard Theorem,  we consider a fixed $K$-quasiregular map $f\: \R^n \ra \St^n=\wR^n$.  
We want to show that $f$ cannot omit a set of points  if their number is sufficienly large depending on $n$ and $K$. 
By considering $f$ followed by a rotation if necessary, we may assume that 
$\infty\in \wR^n$ is among the omitted values. So suppose $f\: \R^n \ra \R^n$ is a $K$-quasiregular 
map omitting the distinct values $a_1, \dots, a_q\in \R^n$. We want to derive a contradiction if $q$ is sufficiently large only depending on $n$ and $K$. 
Note that $f$ is $K$-quasiregular with the same $K$ independently of whether we equip the target $\R^n$ 
with the Euclidean metric or the restriction of the spherical metric on $\wR^n$
to $\R^n$. This follows from the conformal equivalence of these metrics. 

We consider the measure $A$ on $\R^n$ obtained by pulling back the spherical volume 
on $\R^n\sub \wR^n$ to $\R^n$ by $f$. More explicitly, 
\begin{equation}\label{eq:A}
dA(x)= \frac{2^n \det (Df(x))} {(1+|f(x)|^2)^n}\, d\lambda_n(x), 
\end{equation} 
where as before $Df(x)$ is the almost everywhere defined Jacobi matrix of $f$. 

In the following, $\mathcal{A}$ will be as in \eqref{def:A} with $G$ as in \eqref{def:G}. We want to use the omitted values to construct certain non-negative and continuous $\mathcal{A}$-subharmonic functions whose Riesz measure is equal to $A$.

For each $k=1, \dots, q$ we pick a rotation on $\St^n\cong \wR^n$ such that $R_k(a_k)=\infty$. 
Let $v$ be the function from Lemma~\ref{lem:auxfunc} and define $v_k:= v\circ R_k$. Then $v_k$ is a non-negative 
$C^1$-smooth function on 
$\R^n  \backslash  \{a_k\} \sub \St^n \backslash \{a_k\}$. Since $R_k$ is a conformal map and the spherical measure $\sigma_n$ is rotation-invariant, it follows that 
$v_k$ is $n$-subharmonic on $\R^n \backslash \{a_k\}$ with Riesz measure $\sigma_n$ (restriced to
$\R^n  \backslash  \{a_k\}$ to be precise). Actually, $|\nabla v_k|^{n-2}\nabla v_k$ is $C^1$-smooth and 
$$ \diver (|\nabla v_k(y)|^{n-2}\nabla v_k(y))= \frac{2^n } {(1+|y|^2)^n}$$
for $y\in \R^n  \backslash \{a_k\}$. 
Moreover,  $v_k$ is bounded  outside each neighborhood of $\{a_k\}$ with 
 $v_k(x)\to +\infty$  as $x\to a_k$ 
for  $k=1, \dots, q$. 

We  fix  $\delta>0$ such that the Euclidean balls $B(a_1, \delta), \dots, B(a_k, \delta)$ are pairwise disjoint. 
The behavior  of the functions $v_k$ near the singularities  $a_k$   implies  that if we choose $\delta$ small enough, then we can find a constant $L_0>0$ with the following property: if $u_k(y)>L_0$ for some $k=1, \dots, q$, then $y\in B(a_k, \delta)$. In particular, the sets 
$\{ v_k>L_0\}$, $k=1, \dots, q$,
are pairwise disjoint. 

 Now define $u_k:=v_k\circ f$. This is meaningful, because $f$ omits the value $a_k$. It follows from the definition of $L_0$ that 
the sets 
 \begin{equation}\label{L_0} 
\{ u_k>L_0\}, \quad k=1, \dots, q,
\end{equation}
are pairwise disjoint.

\begin{lemma}\label{lem:pullharm}Each function $u_k$ is an unbounded, non-nega\-tive, and continuous $\mathcal{A}$-sub\-har\-monic function with Riesz measure $A$. 
\end{lemma} 

\begin{proof} Essentially,  this follows from the fact that $u=u_k$ is the pull-back of the non-negative and continuous $n$-subharmonic 
function $v=v_k$ which has Riesz measure $\sigma=\sigma_n$, and  $\sigma$ pulls back to $A$. A more general version of the statement is true for arbitrary  $n$-subharmonic functions. For the proof  one has to struggle with regularity issues. Here the argument is   easier, because 
we have sufficient smoothness. 

Consider  $u=v\circ f\ge 0$. It is clear that $u$ is continuous. Since $v$ is $C^1$-smooth and $f$ is in $W^{1,n}_{loc}(\R^n)$, the function $u$ has a weak derivative $\nabla u$ that is locally $L^n$-integrable. We can apply  the chain rule and obtain  
$$\nabla u (x) = Df(x)^t \nabla v(f(x)).$$ Here and in similar equations below this is understood to hold for a.e.\ $x \in \R^n$. Using (\ref{def:G}) we find that
\[
G(x)\nabla u(x)\cdot \nabla u (x)= \det(Df(x))^{2/n}|\nabla v(f(x))|^2.
\]
Inserting this  into  (\ref{def:A}), we obtain 
\[
\cA(x,\nabla u(x))= |\nabla v(f(x))|^{n-2}\det(Df(x))Df(x)^{-1}\nabla v(f(x)).
\]
This  can be rewritten as
\begin{align*}
\cA(x,\nabla u(x))&= \adjoint(Df(x))\left(|\nabla v(f(x))|^{n-2}\nabla v(f(x))\right).\\
&=  \adjoint(Df(x))\left((V\circ f)(x)\right), 
\end{align*}
where 
$$V:=|\nabla v|^{n-2}\nabla v.$$ 
It follows from from Lemma \ref{lem:auxfunc} that
$V$ is a $C^1$-smooth vector field, and that 
\begin{equation}\label{eq:diverV}
 (\diver V)(x)= \frac{2^n } {(1+|x|^2)^n}. 
 \end{equation} 

Now let $\varphi\in C_c^\infty(\R^n)$. Then based on  Lemma~\ref{lem:divergence} we conclude from  \eqref{eq:diverV} and \eqref{eq:A} that 
\begin{align*}-\int \cA_u \cdot \nabla\varphi&=
-\int  [ \adjoint(Df)(V\circ f)]\cdot \nabla \varphi\,  \\
&=\int \varphi  ((\diver V)\circ f) J_f = \int \varphi\, dA.
\end{align*}
So equation (\ref{eq:subweak}) is satisfied with $\mu=A$. Therefore, $u$ is a non-negative continuous $\cA$-subharmonic with Riesz measure $\mu=A$ as claimed.
Obviously, the function $u$ is non-constant (otherwise $\mu=A=0)$, and so it is unbounded by Lemma~\ref{lem:cacccon1}.  
\end{proof}

\begin{corollary}[Meeting superlevel sets] \label{cor:meettract} There exists a constant $C_0>0$ with the following property: if $B=B(a,r)$ with $a\in \R^n$ and $r>0$  is any  ball satisfying  $A(\tfrac12 B)\ge C_0$, then 
$$ B\cap \{u_k>3L_0\}\ne \emptyset $$
for all $k=1, \dots, q$.
\end{corollary} 

In other words, if a ball has sufficiently large $A$-mass, then the twice larger ball with the same center  meets each superlevel set 
$\{u_k>3L_0\}$.  In general, the constant $C_0$ will not only depend on $n$ and $K$, but on other data (such as $L_0$  and the points 
$a_1, \dots, a_q$).  

\begin{proof} An inequality as in Lemma~\ref{lem:Riesz<grow} holds for each function $u_k$ with $\mu=A$ and  a constant $C$ only depending on $n$ and $K$.  So if the  
$A$-mass of  $\tfrac 12 B$ exceeds a large enough constant $C_0$  (independent of the ball), then for each of the functions $u_k$ there exists a point $x\in B$ with $u_k(x)>3L_0$. The claim follows. \end{proof} 

The inequality in Lemma~\ref{lem:Riesz<grow} says that the growth of an $\mathcal{A}$-sub\-har\-monic function 
controls its Riesz measure. One can also prove a similar  inequality
in the other direction. There seems to be no simple  proof of this fact for general 
$\mathcal{A}$-subharmonic functions. We will only need this for our functions
$u_1, \dots , u_q$, where one can prove a related inequality by a  rather simple argument.   To set this up, we fix $a\in \R^n$,  $r>0$, $t\ge 2$, and let 
$B=B(a,r)$ and $R=  \{ x\in \R^n: r<|x-a|<tr\}$. 
For $k=1, \dots, q$ we define 
\begin{equation}\label{M_k}
\displaystyle M_k:= \sup_{x\in B}u_k(x).
\end{equation} and 
 \begin{equation}\label{S_k}S_k  :=R\cap \{M_k/3< u_k <2M_k/3\}. 
 \end{equation} Then $\{u_k\ge M_k\} \cap \overline B\ne \emptyset$. 

If $A(\frac 12 B)\ge C_0$, where $C_0$ is the constant  in Corollary~\ref{cor:meettract}, then  $M_k/3\ge L_0$ for each $k=1,\dots, q$. This implies that $S_k\sub \{u_k>L_0\}$ and so by  \eqref{L_0}  the 
sets $S_k$, $k=1, \dots, q$, are pairwise disjoint. 
 
We consider $S_k$ as a condenser in the ring domain   $R$ with the complementary 
sets $$ S'_k :=R\cap \{ u\ge   2 M_k/3\}\text{ and }   S''_k :=R\cap 
\{  u\le   M_k/3\}. $$
See Figure~\ref{fig} (and the proof of Lemma~\ref{lem:capest} for more discussion).

We define the capacity of $S_k$ (for given $R$) as 
$$ \CAP(S_k) := \inf \biggl\{ \int_{S_k} |\nabla \psi|^n: \psi\in C^\infty_c(\R^n), 
\psi|S'_k\ge 1, \, \psi|  S''_k \le 0 \biggr\}. $$  
Based on an approximation argument, one can here replace 
the class $C^\infty_c(\R^n)$ of test functions  by the larger class of all 
functions in 
$W^{1,n}_{loc}(\R^n)$ with compact support. If $\varphi \in C_c^\infty (\R^n)$ and $\varphi|tB=1$, then   
$\psi=(3u_k/M_k-1)\varphi$  lies in the latter class. It follows that 
\begin{equation}\label{eq:capest}
 \frac{M_k^n}{3^n} \CAP(S_k)\le \int_{S_k} |\nabla u_k|^n.   
  \end{equation} 
  
  \begin{figure}[h!]
\begin{center}
\begin{tikzpicture}[thick, scale=0.6]

\def\ringa{(0,0) circle (2cm) (0,0) circle (6cm)}
\def\ringb{ (-4.5,6) parabola bend (0,0.5) (4.5,6) to (-7,6) parabola bend (0,-0.5) (7,6)}

\begin{scope}[even odd rule]

\clip \ringa;

\fill[fill=gray!50] \ringb;

\end{scope}

\draw[thick][thick] (0,0) circle (2cm);

\draw[thick] (0,0) circle (6cm);

\draw[thick] (-3,6) parabola bend (0,2) (3,6);

\draw[thick] (-4.5,6) parabola bend (0,0.5) (4.5,6);
\draw[thick] (-7,6) parabola bend (0,-0.5) (7,6);

\draw[thick] (0,4) node {$u_k>M_k$};
\draw[thick] (1,-1) node {$B$};
\draw[thick] (3.9,-3.9) node {$R$};
\draw[thick] (2.9,4) node {$S^\prime_k$};
\draw[thick] (3.6,2.2) node {$S_k$};
\draw[thick] (4.2,0.3) node {$S^{\prime\prime}_k$};
\end{tikzpicture}
\end{center}
\caption{The condenser $S_k$.}\label{fig}
\end{figure}
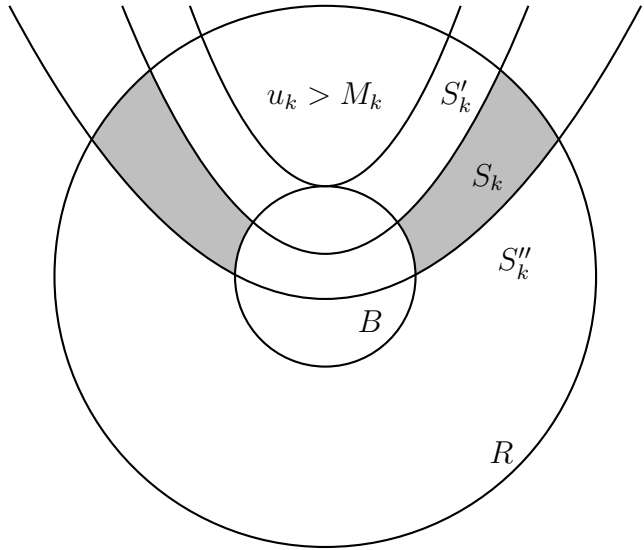

 The sets $S_1, \dots, S_q$ are pairwise 
disjoint, and so they  are crowded together in the ring  $R$.
It is not hard to see that that if $q$ is large, then there exists $k\in \{ 1, \dots, q\}$ such that $\CAP(S_k)>>1$. Actually, we will record a corresponding capacity estimate that gives an essentially optimal bound.

\begin{lemma}[Capacity estimate] \label{lem:capest} With the setup as above, there exists a constant $C_0>0$ such that if $q\ge 2$
and $ A(\frac 12 B)> C_0$, then 
 $$\sum_{k=1}^q  \CAP(S_k) \ge c \log(t)  q^{n/(n-1)}, $$ 
 where $c=c(n)>0$.  
\end{lemma}  

\begin{proof} In this proof all implicit multiplicative constants in inequalities of the form $X\lesssim Y$ only depend on $n$.

By Corollary~\ref{cor:meettract} we can find a constant $C_0>0$ independent of the ball $B=B(a,r)$ such that if 
$ A(\frac 12 B)> C_0$, then the superlevel sets $\{u_k >  M_k/3\}\sub \{u_k>L_0\}$, $k=1, \dots, q$,  are pairwise disjoint. We remark here for later use  that this last statement (and hence the rest of the argument) is also true without the assumption $ A(\frac 12 B)> C_0$ if $r$ is large enough,  because the functions $u_k$ are unbounded and so 
 $M_k/3\ge L_0$ for all $k=1, \dots, q$, if $r>0$ is large. 
 
 By Lem\-ma~\ref{lem:cacccon2}   no component of 
 $\{u_k > 2M_k/3\}$
is bounded. Since  $\{u_k > 2M_k/3\} \supseteq \{u_k\ge  M_k\}$  meets  $\overline B(a,r)$, 
it follows that $S'_k=R\cap \{u_k \ge   2M_k/3\}$  
has non-empty intersection with  each  sphere
$$\Sigma(\rho):= \{ x \in \R^n: |x-a|=\rho\}$$ for 
$r<\rho<tr$.  Now 
$$ S''_k =R\cap \{u_k \le   M_k/3\} \supseteq S'_l $$ for each $l\ne k$, and so $S''_k$ also meets each  sphere
$\Sigma(\rho)$  with 
$r<\rho<tr$ (here it is important that $q\ge 2$ which implies that there exists $l\in \{1, \dots, k\}$ with $k\ne l$).

Now fix $k\in \{1, \dots, q\}$ for the moment, and let $u=u_k$, $M=M_k$, $S=S_k$, $S'=S'_k$, and $S''=S''_k$.  By what we have seen,  for each  fixed $\rho\in (r,tr)$ we can  pick $x_\rho\in S'\cap \Sigma(\rho)$ and $y_\rho\in S''\cap \Sigma(\rho)$. We may assume that $|x_\rho-y_\rho|$ is minimal among all such points.
Then $S\cap  \Sigma(\rho)$ contains a spherical cap  of radius comparable to $|x_\rho-y_\rho|$. 
Hence 
\begin{equation}\label{eq:surfvol}
|x_\rho-y_\rho|^{n-1} \lesssim \sigma_{n-1} (S\cap \Sigma(\rho)). 
\end{equation}

Let $\psi \in C^\infty_c(\R^n)$ be a test function for the condenser $S$, and let 
$\widetilde \psi:= \min\{ 1, \max\{ \psi, 0\}\}$. Then   
$\nabla \widetilde \psi(x) = \nabla \psi(x)  $ for a.e.~$x\in S$ and 
$\nabla \widetilde \psi(x)=0$ for a.e.\ $x\in R\backslash  S$. Hence by Fubini we have 
$$ \int_{\Sigma(\rho)} |\nabla \widetilde \psi|^n\, d\sigma_{n-1} =  \int_{S\cap \Sigma(\rho)} |\nabla  \psi|^n \, d\sigma_{n-1} $$ for a.e.\ $\rho\in (r,tr)$. 

If we apply the Sobolev embedding theorem for the supercitical exponent $p=n$ 
on the $(n-1)$-dimensional sphere $ \Sigma(\rho)$ to the function 
$\tilde \psi$, then we conclude that 
$$ 1\le \widetilde\psi(y_\rho)-\widetilde \psi(x_\rho) \lesssim |x_\rho-y_\rho|^{1-(n-1)/n} 
\biggl(     
 \int_{S\cap \Sigma(\rho)} |\nabla  \psi|^n \, d\sigma_{n-1}
\biggr)^{1/n},   $$ 
and so by  \eqref{eq:surfvol} we obtain   that 
$$\frac{1}{\sigma_{n-1}(S\cap \Sigma(\rho))^{1/(n-1)}}\lesssim 
\int_{S\cap \Sigma(\rho)} |\nabla  \psi|^n \, d\sigma_{n-1} $$
for a.e.\ $\rho \in (r,tr)$. Integrating over $(r,tr)$ we arrive at 
$$  \int_r^{tr} \frac{d\rho }{\sigma_{n-1}(S\cap  \Sigma(\rho))^{1/(n-1)}}\lesssim \int_S |\nabla  \psi|^n,$$
and taking the infimum over all test function $\psi$ we conclude 
\begin{equation}\label{eq:ahlforsbd}
 \int_r^{tr} \frac{d\rho }{\sigma_{n-1}( S\cap\Sigma(\rho))^{1/(n-1)}}\, d\rho\lesssim \CAP(S). 
\end{equation}

This estimate holds for all the condensers $S_1, \dots, S_q$. Since they are pairwise disjoint, we have 
$$ \sum_{k=1}^q \sigma_{n-1}(S_k\cap \Sigma(\rho ))\le 
\sigma_{n-1}(\Sigma(\rho)) \lesssim \rho^{n-1}
$$ for all $\rho \in (r, tr)$.
Now by H\"older's inequality, 
$$
q\le \biggl(   \sum_{k=1}^q 
  \frac{1}{\sigma_{n-1}( \Sigma(\rho)\cap S_k)^{1/(n-1)}}\biggr)^{(n-1)/n} 
\biggl(\sum_{k=1}^q \sigma_{n-1}( \Sigma(\rho)\cap S_k)\biggr)^{1/n}, 
$$ and so 
$$ q^{n/(n-1)}/\rho \lesssim 
 \sum_{k=1}^q 
  \frac{1}{\sigma_{n-1}( \Sigma(\rho)\cap S_k)^{1/(n-1)}}.
  $$
If  we integrate  over $(r,tr)$
with respect to $\rho$, then  the claim  follows from \eqref{eq:ahlforsbd}.  
\end{proof}

The next result will be used in two ways, either when the value of the parameter $t$ tends to infinity or  when $t=2$.

\begin{lemma}[Riesz measure controls growth]\label{lem:grow<Riesz} 
Let $a\in \R^n$, $r>0$,  $t\ge 2$,  $B=B(a,r)$, and $R=  \{ x\in \R^n: r<|x-a|<tr\}$.     For some  $k\in \{1, \dots, q\}$,  set  $u=u_k$,  $M=\displaystyle \sup_{x\in B} u(x)$, and 
$$S=R\cap \{M/3<u<2M/3\}.$$  Then   we have 
\begin{equation}\label{eq:Capest}
( \CAP(S)-C_1)\sup_{x\in B} u(x)^{n-1}    \le C_2A(2tB),
\end{equation}
where $C_1=C_1(n,K)>0$ and $C_2=C_2(n,K)>0$.  
\end{lemma} 

The inequality is trivial if $\CAP(S)\le C_1$. The main point is that if $t$  is large or  $q$ is large, then 
necessarily $\CAP(S_k)>> C_1$ for some $k$ and so we get a non-trivial estimate.  

Before we turn to the proof of the lemma, we want to apply it to verify that the  measure $A$ satisfies the hypotheses of Lemma~\ref{lem:hunt} if $q\ge 2$ (which we may assume). To see  this,  it is enough to show that  $A(\R^n)=\infty$.  We use the remark in the beginning of the proof of 
Lemma~\ref{lem:capest}. Namely, its  conclusion is true 
for $q\ge 2$ and the ball $B=B(0,r)$ if $r>0$ is large enough (independently of the value
$A(\frac 12 B)$). 

It follows that we can fix $t\ge 2$ (independently of $r$ if $r$ is large enough) such that 
for one of the condensers $S_k$ we have $\text{Cap}(S_k)>2C_1$, where $C_1$ is the constant in Lemma~\ref{lem:grow<Riesz}.
Inequality \eqref{eq:Capest} then shows that we must have 
$A(\R^n)=\infty$, because the function $u_k$ is unbounded and so the right hand side in \eqref{eq:Capest} 
can be made arbitrarily large if we choose $r$ large enough. 

 \begin{proof}[Proof of Lemma~\ref{lem:grow<Riesz}]  In this proof all implicit multiplicative constants will only depend on $n$ and $K$. We have 
$$- \int  \mathcal{A}_u\cdot \nabla \varphi =\int \varphi \, dA$$
for all $\varphi\in C_c(\R^n)\cap W^{1,n}_{loc}(\R^n)$. 
We can 
use this identity for  the test function 
$$ \varphi= (M-u)_+\psi^n, $$ where we choose $\psi\in C^\infty_c(\R^n)$ so that it satisfies  
$$0\le \psi \le 1, \ \psi|tB = 1,\  \psi| \R^n\backslash 2tB= 0, \ |\nabla \psi|\lesssim \frac{1}{tr}. $$
 Then $0\le \varphi \le M$, 
$$  \nabla \varphi (x) = -\nabla u(x) \psi(x)^n + n (M-u(x))_+   \nabla \psi(x) \psi^{n-1}(x)$$
for a.e.\ $x\in E$, and $\nabla \varphi (x)=0$ for a.e.\ $x\in \R^n \backslash  E$, where 
 $$E:=  2tB\cap\{ u < M\}\supseteq S.$$  
Hence 
\begin{align*}
\int_E   \psi^n |\nabla u|^n& \lesssim \int_E  \psi^n
 \mathcal{A}_u\cdot \nabla u    \\ &= \int  (M-u)_+\psi^n\, dA+ 
n \int_E   (M-u)_+   \psi^{n-1} \mathcal{A}_u\cdot   \nabla \psi \\
&\lesssim M A(2t B) + M \int_E    \psi^{n-1} |\mathcal{A}_u| \cdot
  |\nabla \psi| \\
& \lesssim M A( 2t B) + M \int_E  \psi^{n-1}\ |\nabla u|^{n-1}    |\nabla \psi| 
\\
& \lesssim M A(2t  B) +M   \biggl(\int_E \psi^n  |\nabla u|^{n}  \biggr)^{(n-1)/n}
\biggl(\int |\nabla \psi|^{n} \biggr)^{1/n}\\
&
 \lesssim M A(2t B) +M \biggl(\int_E \psi^n  |\nabla u|^{n}  \biggr)^{(n-1)/n}.
\end{align*}
In the last step we used that  $\displaystyle \int |\nabla \psi|^{n}\lesssim 1$. 

Now if $\alpha\le \beta + \gamma\alpha^{(n-1)/n}$ for $\alpha,\beta,\gamma\ge 0$, then 
$ \alpha \le \max\{2\beta, 2^n \gamma^n\}$. 
It follows that 
$$ \int_E |\nabla u|^n \psi^n \lesssim \max\{M A(2 t B), M^n\}.
$$
By \eqref{eq:capest} we have 
$$M^n \CAP(S) \lesssim  \int_S |\nabla u|^n \le \int_E |\nabla u|^n \psi^n $$ 
and so we conclude that 
\begin{align*}
M^n \CAP(S) &\lesssim  \max\{M A(2 t  B), M^n\}\\
&\le M A(2tB) + M^n. 
\end{align*}
The claim follows. 
\end{proof}

We are now ready to wrap things up.

\begin{proof}[Proof of Theorem~ \ref{rickpic}] 
We use the setup and the notation introduced in the beginning of this section. 

Let $C_0>0$ be the constant from Corollary~\ref{cor:meettract}, and $D=D(n)>0$ be the constant from Lemma~\ref{lem:hunt}. 
As we have seen, we can apply this lemma to  the measure $A$ if $q\ge 2$ as we may assume.
It follows  that there exists a ball $B=B(a,r)$ such that 
$$ A(\tfrac 12 B) >C_0\text{ and } A(4B) < D A(\tfrac 12 B). $$
On the other hand, by Lemma~\ref{lem:capest} there exists $k\in \{1, \dots, q\}$ such that 
for the condenser $S=S_k$ as defined in \eqref{S_k} with $t=2$, we have 
$$ \CAP(S)\gtrsim q^{1/(n-1)}. $$ We may assume that $q^{1/(n-1)} \ge 2 C_1$. 
Combining Lemmas \ref{lem:Riesz<grow} and \ref{lem:grow<Riesz} for $t=2$ and  $u=u_k$, we obtain  
$$ DA(\tfrac 12 B) \ge A(4B) \gtrsim  q^{1/(n-1)}  \sup_{x\in B}  u(x)^{n-1} \gtrsim q^{1/(n-1)} A(\tfrac 12 B).
$$  In the previous inequalities all implicit multiplicative constants only depend on $n$ and $K$. It follows that the previous inequality is impossible  if   $q$ exceeds a constant only depending on $n$ and $K$. 
\end{proof}

 
\end{document}